\documentclass[reqno, a4paper, 10pt]{amsart}
\usepackage{amssymb}
\usepackage{mathrsfs}
\usepackage{amssymb, url, color, graphicx, amscd, mathrsfs}
\usepackage[colorlinks=true, bookmarks=true, pdfstartview=FitH, pagebackref=true, linktocpage=true, linkcolor = magenta, citecolor = blue]{hyperref}
\usepackage{graphicx}
\usepackage{times}
\newtheorem{theorem}{\bf Theorem}[section]
\newtheorem{define}[theorem]{\bf Definition}
\newtheorem{corollary}[theorem]{\bf Corollary}
\newtheorem{lemma}[theorem]{\bf Lemma}

\newtheorem*{theorem*}{Theorem}

\numberwithin{equation}{section}

\newcommand{\RR}{\mathbb R}

 \DeclareMathOperator{\Ric}{Ric}

\allowdisplaybreaks
\begin{document}

\title[Vanishing properties of $p$-harmonic $\ell$-forms on Riemannian
manifolds]{Vanishing properties of $p$-harmonic $\ell$-forms on
Riemannian manifolds}
\def\cfac#1{\ifmmode\setbox7\hbox{$\accent"5E#1$}\else\setbox7\hbox{\accent"5E#1}\penalty 10000\relax\fi\raise 1\ht7\hbox{\lower1.0ex\hbox to 1\wd7{\hss\accent"13\hss}}\penalty 10000\hskip-1\wd7\penalty 10000\box7 }

\author[N.T. Dung]{Nguyen Thac Dung}
\address[N.T. Dung]{Department of mathematics, College of Science\\ Vi\^{e}t Nam National University, Ha N\^{o}i, Vi\^{e}t Nam}
\email{\href{mailto: N.T. Dung
<dungmath@gmail.com>}{dungmath@gmail.com}}

\author[P.T. Tien]{Pham Trong Tien}
\address[P.T. Tien]{Department of mathematics, College of Science\\ Vi\^{e}t Nam National University, Ha N\^{o}i, Vi\^{e}t Nam}
\email{\href{mailto: Pham Tien <phamtien@mail.ru>}{phamtien@mail.ru,
phamtien@vnu.edu.vn}}

\begin{abstract}
In this paper, we show several vanishing type theorems for
$p$-harmonic $\ell$-forms on Riemannian manifolds ($p\geq2$). First of all, we consider complete non-compact immersed submanifolds $M^n$ of ${N}^{n+m}$
with flat normal bundle, we prove that any $p$-harmonic $\ell$-forms on $M$ is trivial if $N$ has pure curvature tensor and $M$ satisfies some geometric condition. Then, we obtain a vanishing theorem on Riemannian manifolds with weighted
Poincar\'{e} inequality. Final, we investigate complete simply connected, locally
conformally flat Riemannian manifolds $M$ and point out that there is no nontrivial $p$-harmonic $\ell$-form on $M$ provided that $\operatorname{Ric}$ has suitable bound.


\end{abstract}
\maketitle \vskip0.4cm
\section{Introduction}
Suppose that $M$ is a complete noncompact oriented Riemannian
manifold of dimension $n$. At a point $x\in M$, let
$\left\{\omega_1, \ldots, \omega_n\right\}$ be a positively oriented
orthonormal coframe on $T_x^{*}(M)$, the Hodge star operator acting
on the space of smooth $\ell$-forms $\Lambda^{\ell}(M)$ is given by
$$ *(\omega_{i_1}\wedge\ldots\wedge\omega_{i_\ell})=\omega_{j_1}\wedge\ldots\wedge\omega_{j_{n-\ell}}, $$
where $j_1, \ldots, j_{n-\ell}$ is selected such that
$\{\omega_{i_1},\ldots,\omega_{i_\ell},
\omega_{j_1},\ldots,\omega_{j_{n-\ell}}\}$ gives a positive
orientation. Let $d$ is the exterior differential operator, so its
dual operator $\delta$ is defined by
$$ \delta=* d *. $$
Then the Hogde-Laplace-Beltrami operator $\Delta$ acting on the
space of smooth $\ell$-forms $\Omega^{\ell}(M)$ is of form
$$ \Delta =-(\delta d+d\delta). $$
When $M$ is compact, it is well-known that the space of harmonic
$\ell$-forms is isomorphic to its $\ell$-th de Rham cohomology
group. This is not true when $M$ is non-compact but the theory of
$L^{2}$ harmonic forms still have some interesting applications. For
further results, we refer the reader to \cite{Car1, Car2}. In
\cite{Li}, Li studied Sobolev inquality on spaces of harmonic
$\ell$-forms on compact Riemannian manifolds. He gave estimates of
the bottom of $\ell$-spectrum and proved that the space of harmonic
$\ell$-forms is of finite dimension provided that the Ricci
curvature is bounded from below. In \cite{Tanno96}, Tanno studied
$L^{2}$ harmonic $\ell$-form on complete orientable stable minimal
hypersurface $M$ immersed in the Euclidean space $\RR^{n+1}$. He
showed that there are no non trivial $L^{2}$ harmonic $\ell$-forms
on $M$ if $n\leq 4$. Later, Zhu generalized Tanno's results to
manifolds with non-negative isotropic curvature (see
\cite{Zhu2011}). He also proved in \cite{Zhu01} that the Tanno's
results are true if $M^{n}, n\leq 4$ be a complete noncompact
strongly stable hypersurface with constant mean curvature in
$\RR^{n+1}$ or $\mathbb{S}^{n+1}$. Recently, in \cite{Lin15a}, Lin
investigated spaces of $L^{2}$ harmonic $\ell$-forms
$H^{\ell}(L^{2}(M))$ on submanifolds in Euclidean space with flat
normal bundle. Assumed that the submanifolds are of finite total
curvature, Lin showed that the space $H^{\ell}(L^{2}(M))$ has finite
dimension (See also \cite{Zhu2015} for the case $\ell=2$). For further
results in this direction, we refer to \cite{Lin15b, Lin15c, Lin15d,
Seo, Zhu01, Zhu2011} and the references therein. It is also worth
to notice that the main tools to study the spaces of harmonic
$\ell$-forms are the Bochner formula and refined Kato type
inequalities. In 2000, Calderbank et al. gave very general forms of
Kato type inequalities in \cite{Calderbank2000}. Then in
\cite{XDWang2002}, Wang used them to prove a vanishing property of
the space of $L^{2}$ harmonic $\ell$-forms on convex cocompact
hyperbolic manifolds. Later, in \cite{WanXin2004}, Wan and Xin
studied $L^{2}$ harmonic $\ell$-forms on conformally compact
manifolds with a rather weak boundary regularity assumption.
Recently, in \cite{CiboZhu2012}, Cibotaru and Zhu introduced a proof
of the mentioned results from \cite{Calderbank2000} avoiding as much
as possible representation theoretic technicalities. The refined
Kato type inequalities they obtained also refined the ones used in
\cite{XDWang2002, WanXin2004}.

On the other hand, the $p$-Laplacian operator on a Riemannian
manifold $M$ is defined by
$$\Delta_p u := \operatorname{Div}(|\nabla u|^{p-2} \nabla u)$$
for any function $u\in W_{loc}^{1,p} (M)$ and $p> 1$, which arises
as the Euler-Lagrange operator associated to the $p$-energy
functional
$$E_p (u):=\int_M |\nabla u|^p.$$
Therefore, if $u$ is a smooth $p$-harmonic function then $du$ is a
$p$-harmonic $1$-forms. We refer the reader to \cite{Mo2007, KL2009}
for the connection between $p$-harmonic functions and the inverse
mean curvature flow. Motivated by the above beautiful relationship
between the space of harmonic $\ell$-forms and the geometry of
Riemannian manifolds, it is very natural for us to study the
geometric structures of Riemannian manifolds by using the vanishing
properties of $p$-harmonic $\ell$-forms with finite $L^{q}$ energy
for some $p\geq2$ and $q\geq0$.

Recall that an $\ell$-form $\omega$ on a Riemannian manifold $M$ is
said to be $p$-harmonic if $\omega$ satisfies $d\omega=0$  and
$\delta(|\omega|^{p-2}\omega)=0$. When $p=2$, a $p$-harmonic
$\ell$-form is exactly a harmonic $\ell$-form. Some vanishing
propeties of the space of $p$-harmonic $\ell$-forms are given by X.
Zhang in \cite{Zhang}. In particular, Zhang showed that there are no
nontrivial $p$-harmonic $1$-forms in $L^{q}(M), q>0$ if the Ricci
curvature on $M$ is nonnegative. Motivated by Zhang's results, in
\cite{CQS10}, Chang-Guo-Sung proved that any bounded set of
$p$-harmonic $1$-forms in $L^{q}(M), 0<q<\infty$, is relatively
compact with respect to the uniform convergence topology. Recently,
it is showed that the set of $p$-harmonic $1$-forms has closed
relationship with the connectedness at infinity of the manifold, in
particularly, with $p$-nonparabolic ends. In \cite{DS}, the first
author and Seo studied the connectedness at infinity of complete
submanifolds by using the theory of $p$-harmonic function. For
lower-dimensional cases, we proved that if $M$ is a complete
orientable noncompact hypersurface in $\RR^{n+1}$ and
$\delta$-stability inequality holds on $M$, then $M$ has at most one
$p$-nonparabolic end. It was also proved that if $M^n$ is a complete
noncompact submanifold in $\mathbb{R}^{n+k}$ with sufficiently small
$L^n$-norm of the traceless second fundamental form, then $M$ has at
most one $p$-nonparabolic end. For the reader's convenience, let us
recall a definition of the $p$-nonparabolic ends. Let $E \subset M$
be an end of $M$, namely, $E$ is an unbounded connected component of
$M\setminus \Omega$ for a sufficiently large compact subset $\Omega
\subset M$ with smooth boundary. As in usual harmonic function
theory, we define the $p$-parabolicity and $p$-nonparabolicity of
$E$ as follows (see \cite{CCW} and the references therein):
\begin{define}
{\rm An end $E$ of the Riemannian manifold $M$ is called {\it
$p$-parabolic} if for every compact subset $K \subset \overline{E}$
$${\rm cap}_p (K, E):= \inf \int_E |\nabla f|^p = 0$$
where the infimum is taken among all $f\in C^{\infty}_0
(\overline{E})$ such that $f\geq 1$ on $K$. Otherwise, the end $E$
is called {\it $p$-nonparabolic}.}
\end{define}
The first main result in this paper is the below theorem.
\begin{theorem}\label{main1}
Let $M^n\ (n \geq 3)$ be a complete non-compact immersed submanifold
of $N^{n+m}$ with flat normal bundle. Assume that $N^{n+m}$ has pure curvature tensor and the
$(1,n-1)$-curvature of $N^{n+m}$ is not less than $-k$. If one of
the following conditions
\begin{itemize}
\item[1.]
$$
|A|^2 \leq \dfrac{n^2 |H|^2 - 2k}{n-1},
$$
\item[2.]
$$
\dfrac{n^2 |H|^2 - 2k}{n-1} < |A|^2 \leq \dfrac{n^2 |H|^2}{n-1}
\text{ and }
 \lambda_1(M) > \dfrac{kp^2(n-1)}{4(p-1)(n+p-2)},
$$
\item[3.]
the total curvature $\|A\|_n$ is bounded by
$$
\|A\|_n^2 < \min \left\{ \dfrac{n^2}{(n-1)C_S},\
\dfrac{2}{(n-1)C_S}\left[ \dfrac{4(p-1)(n+p-2)}{p^2(n-1)} -
\dfrac{k}{\lambda_1(M)} \right] \right\}
$$
\item[4.] $\sup_M |A|$ is bounded and the fundamental tone satisfies
 $$
\lambda_1(M) > \dfrac{p^2(n-1)(2k + (n-1)\sup_M
|A|^2)}{8(p-1)(n+p-2)},
 $$
\end{itemize}
holds true then every $p$-harmonic $L^{p}$ $1$-form  on $M$ is
trivial. Therefore, $M$ has at most one $p$-nonparabolic end. Here
$C_S$ is the Sobolev constant which depends only on $n$ (see Lemma
\ref{sb-ie}).
\end{theorem}
This results generalizes a work of the first author in \cite{Dung2017}
On the other hand, we consider $p$-harmonic $\ell$-forms on
Riemannian manifolds with weighted Poincar\'{e} inequality. Recall
that let $(M^{n}, g)$ be a Riemannian manifold of dimension $n$ and
$\rho\in{\mathcal{C}}(M)$ be a positive function on $M$. We say that
$M$ has a weighted Poincar\'{e} inequality, if
\begin{equation}\label{weighted}
\int_M\rho\varphi^{2}\leq\int_M|\nabla\varphi|^{2}
\end{equation}
holds true for any  smooth function
$\varphi\in{\mathcal{C}}^{\infty}_0(M)$ with compact support in $M$.
The positive function $\rho$ is called the weighted function. It is
easy to see that if the bottom of the spectrum of Laplacian
$\lambda_1(M)$ is positive then $M$ satisfies a weighted
Poincar\'{e} inequality with $\rho\equiv\lambda_1$. Here
$\lambda_1(M)$ can be characterized by variational principle
$$ \lambda_1(M)=\inf\left\{\frac{\int_M|\nabla\varphi|^{2}}{\int_M\varphi^{2}}: \varphi\in{\mathcal{C}}_0^{\infty}(M)\right\}. $$
When $M$ satisfies a weighted Poincar\'{e} inequality then $M$ has
many interesting properties concerning topology and geometry. For
example, in \cite{vieira}, Vieira obtained vanishing theorems for
$L^{2}$ harmonic $1$ forms on complete Riemannian manifolds
satisfying a weighted Poincar\'{e} inequality and having a certain
lower bound of the bi-Ricci curvature. His theorems are improvement
of Li-Wang's and Lam's results (see \cite{Lam, LW1, LW2}). Moreover,
some applications to study geometric structures of minimal
hypersurfaces also are given. We refer to \cite{ChenSung2009,
DungSung2014} and the refrences therein for further results on the
vanishing property of the space of harmonic $\ell$-forms. In the
nonlinear setting, Chang-Chen-Wei studied $p$-harmonic functions
with finite $L^{q}$ energy  in \cite{CCW}, and proved some vanishing
type theorems on Riemannian manifolds satisfying a weighted
Poincar\'{e} inequality. Later, Sung-Wang, Dat and the first author
used theory of $p$-harmonic functions to show some interesting
rigidity propeties of Riemannian manifolds with maximal
$p$-spectrum. (see \cite{DD, SW}). In this paper, we will investigate manifolds with weighted Poincar\'{e} inequality and prove some vanishing results for $p$-harmonic $\ell$-forms on the manifolds. Our results can be considered as a generalization of Vieira's and the first author's results (see \cite{Dung2017}). Finally,  we are also interested in locally conformally flat
Riemannian manifolds, our theorem is the
following vanishing property.
\begin{theorem}\label{main2}
Let $(M^n,g), \ n \geq 3$, be an $n$-dimensional complete, simply
connected, locally conformally flat Riemannian manifold. If one of
the following conditions
\begin{itemize}
\item[1.]
\begin{equation*}
 \|T\|_{n/2} + \dfrac{\|R\|_{n/2}}{\sqrt{n}} < \dfrac{4\left(p-1 + \min\left\{1, \frac{(p-1)^2}{n-1}\right\}\right)}{Sp^2}.
\end{equation*}
\item[2.] the scalar curvature $R$ is nonpositive and
$$
K_{p,n}: = \dfrac{p-1+ \min\left\{1, \frac{(p-1)^2}{n-1}\right\}}{p^2} -
\dfrac{n-1}{\sqrt n(n-2)} > 0
$$
and
$$
\|T\|_{n/2} < \dfrac{4K_{p,n}}{S}=4K_{p,n}\mathcal{Y}(\mathbb{S}^n).
$$
\end{itemize}
holds true, then every $p$-harmonic $1$-form with finite $L^{p} (p\geq2)$
norm on $M$ is trivial, and $M$ must has at most one
$p$-nonparabolic end. Here $T$ is the traceless Ricci tensor and $S$
is the constant given in Lemma \ref{sb-lem-1}.
\end{theorem}
The rest of this paper is organized as follows. In Section 2, we recall some basic notations and useful fact on theory of smooth $\ell$-forms. Then, we study $p$-harmonic $\ell$-forms in submanifolds of $N^{n+m}$ with flat normal bundle and pure curvature tensor. We will give a proof of Theorem \ref{main1} in this section. In Section 4, we derive some vanishing properties for $p$-harmonic $\ell$-forms on manifolds with weighted Poincar\'{e} inequality. Finally, in Section 5, we consider locally conformally flat Riemannian manifolds and give a proof of Theorem \ref{main2}.
\section{Preliminary notations}
In this paper for $n \geq 3,\ 1 \leq \ell \leq n-1$ and $p \geq 2,\
q \geq 0$, we denote
$$C_{n, \ell}:= \max\{\ell, n-\ell\}$$
and
$$
A_{p,n,\ell} =
\begin{cases}
1+\dfrac{1}{\max{\ell, n-\ell}}, & \text{ if } p =2 \\
1+\dfrac{1}{(p-1)^2}\min\left\{1, \dfrac{(p-1)^2}{n-1}\right\}, & \text{ if } p > 2 \text{ and } \ell =1 \\
1, & \text{ if } p > 2 \text{ and } 1 < \ell \leq n-1. \\
\end{cases}
$$
Hence,
$$
(p-1)^2(A_{p,n,\ell}-1) =
\begin{cases}
\dfrac{1}{\max\{\ell, n-\ell\}},&\text{ if } p=2 \text{ and } 1 \leq \ell \leq n-1 \\
\min\left\{1, \dfrac{(p-1)^2}{n-1}\right\}, & \text{ if } p > 2 \text{ and } \ell =1 \\
0, & \text{ if } p > 2 \text{ and } 1 < \ell \leq n-1.
\end{cases}
$$

We will use the following Sobolev inequality.

\begin{lemma}[\cite{Spruck}]\label{sb-ie}
Let $M^n\ (n \geq 3)$ be an $n$-dimensional complete submanifold in
a complete simply-connected manifold with nonpositive sectional
curvature. Then for any $f \in W_0^{1,2}(M)$ we have
$$
\left( \int_M |f|^{\frac{2n}{n-2}} dv \right)^{\frac{n-2}{n}} \leq
C_S \int_M \left(|\nabla f|^2 + |H|^2f^2\right)dv,
$$
where $C_S$ is the Sobolev constant which depends only on $n$.
\end{lemma}
An important ingredient in our methods is the following refined Kato
inequality. In order to state the inequality, let us recall some
notations. Suppose that $\left\{e_1, \ldots, e_n\right\}$ is an
orthonormal basic of $\RR^n$. Let
$\theta_1:\Lambda^{\ell+1}\RR^n\to\RR\otimes\Lambda^{\ell}\mathbb{R}^n$ given by
$$ \theta_1(v_1\wedge\ldots\wedge v_{\ell+1})=\frac{1}{\sqrt[]{\ell+1}}\sum\limits_{j=1}^{\ell+1}(-1)^{j}v_j\otimes v_1\wedge\ldots\wedge\widehat{v_j}\wedge\ldots\wedge v_{\ell+1} $$
and $\theta_2: \Lambda^{\ell-1}\RR^n\to\RR\otimes\Lambda^{\ell}\RR^n$ by
$$ \theta_2(\omega)=-\frac{1}{\sqrt[]{n+1-\ell}}\sum\limits_{j=1}^{n}e_j\otimes(e_j\wedge\omega). $$
\begin{lemma}[\cite{Li:geometricanalysis, Calderbank2000, DS}]\label{kt-ine}
For $p\geq2, \ell\geq1$, let $\omega$ be an $p$-harmonic $\ell$-form on
a complete Riemannian manifold $M^n$. The following inequality holds
\begin{equation}\label{kato}
\left| \nabla \left( |\omega|^{p-2}\omega \right) \right|^2 \geq
A_{p, n , \ell} \left| \nabla |\omega|^{p-1} \right|^2.
\end{equation}
Moreover, when $p=2, \ell>1$ then the equality holds if and only if
there exixts a $1$-form $\alpha$ such that
$$ \nabla\omega=\alpha\otimes\omega-\frac{1}{\sqrt[]{\ell+1}}\theta_1(\alpha\wedge\omega)+\frac{1}{\sqrt[]{n+1-\ell}}\theta_2(i_\alpha\omega). $$
\end{lemma}
\begin{proof}The inequality \eqref{kato} is well-known when $\ell=1, p=2$, for example, see \cite{Li:geometricanalysis}.
When $\ell=1, p>2$, the inequality \eqref{kato} was proved by Seo and
the first author in \cite{DS}. Note that when $p=2, \ell>1$, the
inequality \eqref{kato} was proved by Calderbank et al. (see
\cite{Calderbank2000}) but we refer to \cite{CiboZhu2012} in a way
convenient for our purpose without introducing abstract notation.

Finally, when $p>2, \ell>1$, the inequality \eqref{kato} is standard
(see \cite{Calderbank2000}).
\end{proof}
Note that if $M$ is not complete, then \eqref{kato} is still true
provided that $\omega$ is a harmonic field, see \cite{CiboZhu2012}
for further discussion.

Suppose that $M$ is a complete noncompact Riemannian manifold. Let
$\left\{e_1, \ldots, e_n\right\}$ be a local orthonormal frame on
$M$ with dual coframe $\left\{\omega_1, \ldots, \omega_n\right\}$.
Given an $\ell$-form $\omega$ on $M$, the Weitzenb\"{o}ck curvature
operator $K_\ell$ acting on $\omega$ is defined by
$$ K_\ell=\sum\limits_{j.k=1}^{n}\omega_{k}\wedge i_{e_j}R(e_k, e_j)\omega. $$
Using the Weitzenb\"{o}ck curvature operator, we have the following
Bochner type formula for $\ell$-forms.
\begin{lemma}[\cite{Li:geometricanalysis}]\label{bochner}
Let $\omega=\sum_{I}a_I\omega_I$ be a $\ell$-form on $M$. Then
$$\begin{aligned}
 \Delta|\omega|^{2}
&=2\left\langle\Delta\omega, \omega\right\rangle+2|\nabla\omega|^{2}+2\left\langle E(\omega), \omega\right\rangle\\
&=2\left\langle\Delta\omega,
\omega\right\rangle+2|\nabla\omega|^{2}+2K_{\ell}(\omega, \omega)
\end{aligned}$$
where $E(\omega)=\sum\limits_{j, k=1}^{n}\omega_{k}\wedge
i_{e_j}R(e_k, e_j)\omega$.
\end{lemma}

In order to estimate the $K_\ell$, we need to define a new curvature
which is appear naturally as a component of the Weitzenb\"{o}ck
curvature operator (see \cite{Lin15a, Lin15c}).
\begin{define}{\rm
Let $M^n$ be a complete immersed submanifold in a Riemannian
manifold $N^{n+m}$ with flat normal bundle. Here the submanifold $M$ is said to have flat normal bundle if the normal connection of $M$ is flat, namely
the components of the normal curvature tensor of $M$ are zero. For any point $x\in N^{n+m}$, choose an
orthonormal frame $\left\{e_i, \ldots, e_n\right\}_{i=1}^{n+m}$ of the tangent
space $T_xN$ and define
$$ \overline{R}^{(\ell, n-\ell)}([e_{i_1}\ldots, e_{i_n}])=\sum\limits_{k=1}^{\ell}\sum\limits_{h=\ell+1}^{n}\overline{R}_{i_ki_hi_ki_h} $$
for $1\leq \ell\leq n-1$, where the indices $1\leq {i_1}, \ldots
i_n\leq n+m$ are distinct with each other. We call
$\overline{R}^{(\ell, n-\ell)}([e_{i_1}\ldots, e_{i_n}])$ the $(\ell,
n-\ell)$-curvature of $N^{n+m}$. }
\end{define}
It is easy to see that if $\ell=1$ then the $(1, n-1)$-curvature becomes $(n-1)$-th Ricci curvature (\cite{Lin15c}). 
Assume that $N$ has pure curvature tensor, namely for every $p\in N$ there is an orthonormal basis $\{e_1, \ldots, e_n\}$ of the tangent space $T_pN$ such that $R_{ijrs}:= \left\langle R(e_i, e_j)e_r, e_s\right\rangle = 0$ whenever the set $\{i,j,r,s\}$ contains more than two elements. Here $R_{ijrs}$ denote the curvature tensors of $N$. It is worth to notice that
all 3-manifolds and conformally flat manifolds have pure curvature tensor. It was proved in \cite{Lin15a} that
$$ K_\ell(\omega, \omega)\geq \frac{1}{2}(n^{2}|H|^{2}-C_{n, \ell}|A|^{2})+\inf\limits_{i_1,\ldots, i_n}\overline{R}^{(\ell, n-\ell)}([e_{i_1}\ldots, e_{i_n}])|\omega|^{2}.$$

Finally, to prove the vanishing property of $p$-harmonic $\ell$-form, we use the following useful estimates. Since we can not find them in the literature, we will give a detail of discussion.
\begin{lemma}\label{L32}
For any closed $\ell$-form $\omega$ and
$\varphi\in\mathcal{C}^\infty(M)$, we have
$$|d(\varphi\omega)|=|d\varphi\wedge \omega|\leq |d\varphi|\cdot|\omega|.$$
\end{lemma}
\begin{proof}Let $\{X_i\}$ be a local orthonormal frame and $\{dx_i\}$ is the dual coframe. Since $\omega$ is closed, we have $d(\varphi\omega)=d\varphi\wedge \omega$. Suppose that
$$\omega=\sum\limits_{|I|=\ell}\omega_Idx^I=\sum\limits_{|K|=\ell-1}\sum\limits_{j=1}^n\omega_{jK}dx_j\wedge dx^K,$$
where $\omega_{jK}=0$ if $j\in K$. Therefore, denote
$\varphi_i=\nabla_{X_i}\varphi$, we have
$$d\varphi\wedge\omega=\sum\limits_{|K|=\ell-1}\sum\limits_{i\not=j. i,j\not\in K}^n\varphi_i\omega_{jK}dx_i\wedge dx_j\wedge dx^K.$$
Observe that, for any $a_i, b_i\in\mathbb{R}, i=\overline{1,n}$
$$\sum\limits_{i<j}(a_ib_j-a_jb_i)^2+\left(\sum\limits_{i=1}^na_ib_i\right)^2=\left(\sum\limits_{i=1}^na_i^2\right)\left(\sum\limits_{i=1}^nb_i^2\right),$$
we infer
$$\begin{aligned}
|d\varphi\wedge\omega|^2
&=\sum\limits_{|K|=\ell-1}\sum\limits_{i<j, i,j\not\in K}(\varphi_i\omega_{jK}-\varphi_j\omega_{iK})^2\\
&\leq \sum\limits_{|K|=\ell-1}\sum\limits_{i<j}(\varphi_i\omega_{jK}-\varphi_j\omega_{iK})^2\\
&\leq
\sum\limits_{|K|=\ell-1}\left(\sum\limits_{i=1}^n\varphi_i^2\right)\left(\sum\limits_{j=1}^n\omega_{jK}\right)=|d\varphi|^2|\omega|^2.
\end{aligned}$$
The proof is complete.
\end{proof}
\section{$p$-harmonic $\ell$-forms in submanifolds of $N^{n+m}$ with flat normal bundle}
\setcounter{equation}{0}
\begin{theorem}\label{lin-thm}
Let $M^n\ (n \geq 3)$ be a complete non-compact immersed submanifold
of $N^{n+m}$ with flat normal bundle. Assume that $N$ has pure curvature tensor and the
$(\ell,n-\ell)$-curvature of $N^{n+m}$ is not less than $-k$ for $1 \leq
\ell \leq n-1$. If one of the following conditions
\begin{itemize}
\item[1.]
$$
|A|^2 \leq \dfrac{n^2 |H|^2 - 2k}{C_{n, \ell}}, \quad Vol(M)=\infty.
$$
\item[2.]
$$
\dfrac{n^2 |H|^2 - 2k}{C_{n, \ell}} < |A|^2 \leq \dfrac{n^2
|H|^2}{C_{n, \ell}} \text{ and }
 \lambda_1 > \dfrac{kQ^2}{4(Q-1+(p-1)^2(A_{p,n,\ell}-1))},
$$
\item[3.]
the total curvature $\|A\|_n$ is bounded by
$$
\|A\|_n^2 < \min \left\{ \dfrac{n^2}{C_{n, \ell}C_S},\
\dfrac{2}{C_{n, \ell}C_S}\left[ \dfrac{4(Q-1+(p-1)^2(A_{p,n,\ell}-1))}{Q^2} - \dfrac{k}{\lambda_1(M)} \right]
\right\},
$$
\item[4.] $\sup_M |A|$ is bounded and the fundamental tone satisfies
 $$
\lambda_1(M) > \dfrac{Q^2(2k + C_{n, \ell} \sup_M |A|^2)}{8(Q-1+(p-1)^2(A_{p,n,\ell}-1))},
 $$
\end{itemize}
holds true then every $p$-harmonic $\ell$-form with finite $L^Q, (Q\geq2)$ on $M$ is
trivial.
\end{theorem}
\begin{proof}
Let $M_+:=M\setminus\{x\in M, \omega(x)=0\}$. Let $\omega$ be any $p$-harmonic $\ell$-form with finite $L^Q$ norm. Apply the Bochner formula to the form $|\omega|^{p-2}\omega$
we obtain, on $M_+$
$$ \begin{aligned}
\frac{1}{2}\Delta|\omega|^{2(p-1)}
&=|\nabla(|\omega|^{p-2}\omega)|^{2}-\left\langle(\delta d+d\delta)(|\omega|^{p-2}\omega), |\omega|^{p-2}\omega\right\rangle +K_{\ell}(|\omega|^{p-2}\omega, |\omega|^{p-2}\omega)\\
&=|\nabla(|\omega|^{p-2}\omega)|^{2}-\left\langle\delta
d(|\omega|^{p-2}\omega), |\omega|^{p-2}\omega\right\rangle
+|\omega|^{2(p-2)}K_{\ell}(\omega, \omega)
\end{aligned} $$
where we used $\omega$ is $p$-harmonic in the second equality. This
can be read as
$$ |\omega|^{p-1}\Delta|\omega|^{p-1}=\left(|\nabla(|\omega|^{p-2}\omega)|^{2}-|\nabla|\omega|^{p-1}|^{2}\right)-|\omega|^{p-2}\left\langle\delta d(|\omega|^{p-2}\omega), \omega\right\rangle +|\omega|^{2(p-2)}K_{\ell}(\omega, \omega). $$
By Kato type inequality, we infer
\begin{equation}\label{duti1}
|\omega| \Delta |\omega|^{p-1} \geq
(p-1)^{2}(A_{p,n,\ell}-1)|\omega|^{p-2}|\nabla|\omega||^{2}-
\left\langle \delta d (|\omega|^{p-2}\omega), \omega \right\rangle +
K_{\ell} |\omega|^p.
\end{equation}
Hence,
\begin{equation*}
|\omega|^{q+1} \Delta |\omega|^{p-1} \geq
(p-1)^{2}(A_{p,n,\ell}-1)|\omega|^{p+q-2}|\nabla|\omega||^{2}-
\left\langle \delta d (|\omega|^{p-2}\omega), |\omega|^q \omega
\right\rangle + K_{\ell} |\omega|^{p+q}.
\end{equation*}
We choose cutoff function $\varphi \in \mathcal
C^{\infty}_0(M_+)$ then
multiplying both sides of the above inequality by $\varphi^2$, we
obtain
$$\begin{aligned}
\int_{M_+} \varphi^2 |\omega|^{q+1} \Delta |\omega|^{p-1} \geq
& (p-1)^{2}(A_{p,n,\ell}-1)\int_{M_+}\phi^{2}|\omega|^{p+q-2}|\nabla|\omega||^{2}\\
&- \int_{M_+} \left\langle \delta d (|\omega|^{p-2}\omega),
\varphi^2|\omega|^q \omega \right\rangle + \int_{M_+} K_{\ell} \varphi^2
|\omega|^{p+q}.
\end{aligned}$$
by integrating by part, this implies
\begin{align}
\int_{M_+} \left\langle \nabla(\varphi^2 |\omega|^{q+1}), \nabla
|\omega|^{p-1}\right\rangle
\leq &-(p-1)^{2}(A_{p,n,\ell}-1)\int_{M_+}\phi^{2}|\omega|^{p+q-2}|\nabla|\omega||^{2}\notag\\
&+ \int_{M_+} \left\langle d (|\omega|^{p-2}\omega),
d(\varphi^2|\omega|^q \omega) \right\rangle -\int_{M_+} K_{\ell} \varphi^2
|\omega|^{p+q}. \label{m-f}
\end{align}
Since the $({\ell, n-\ell})$-curvature of $N^{n+m}$ is not less than $-k$,
we have
$$
K_{\ell} \geq \dfrac{1}{2} (n^2 |H|^2 - \max\{\ell, n-\ell\}|A|^2) - k.
$$
Obviously,
\begin{align}
\int_{M_+} \left\langle \nabla(\varphi^2 |\omega|^{q+1}), \nabla
|\omega|^{p-1}\right\rangle
=& 2(p-1) \int_{M_+} \varphi |\omega|^{p+q-1}\left\langle \nabla \varphi, \nabla |\omega| \right\rangle \nonumber \\
& + (q+1)(p-1) \int_{M_+} \varphi^2 |\omega|^{p+q-2}|\nabla |\omega||^2.
\label{ls-f}
\end{align}
Note that for any close $\ell$-form $\omega$ and smooth function $f$, it is
proved in Lemma \ref{L32} that
$$
d(f \wedge \omega)=df\wedge\omega \leq |df|.|\omega|.
$$
Therefore,
\begin{align}
\int_{M_+} \left\langle d (|\omega|^{p-2}\omega), d(\varphi^2|\omega|^q \omega) \right\rangle  =& \int_{M_+} \left\langle d(|\omega|^{p-2}) \wedge \omega, d(\varphi^2|\omega|^q) \wedge \omega \right\rangle \nonumber \\
\leq &\int_{M_+} \left|d(|\omega|^{p-2}) \wedge \omega|.|d(\varphi^2|\omega|^q) \wedge \omega \right| \nonumber \\
\leq &\int_{M_+} \left| \nabla(|\omega|^{p-2})\right||\omega|.\left|\nabla(\varphi^2|\omega|^q) \right| |\omega| \nonumber \\
\leq &(p-2)q \int_{M_+} \varphi^2 |\omega|^{p+q-2}\left| \nabla |\omega| \right|^2 \nonumber \\
& + 2(p-2) \int_{M_+} \varphi |\omega|^{p+q-1}\left| \nabla|\omega|
\right| |\nabla \varphi|. \label{d-f}
\end{align}

$1.$ If $|A|^2 \leq \dfrac{n^2 |H|^2 - 2k}{C_{n, \ell}}$, then $K_{\ell} \geq
0$. Therefore, from \eqref{m-f} we obtain
$$\begin{aligned}
\int_{M_+} \left\langle \nabla(\varphi^2 |\omega|^{q+1}), \nabla
|\omega|^{p-1}\right\rangle \leq
& -(p-1)^{2}(A_{p,n,\ell}-1)\int_{M_+}\varphi^2|\omega|^{p+q-2}|\nabla|\omega||^{2}\\
&+ \int_{M_+} \left\langle d (|\omega|^{p-2}\omega),
d(\varphi^2|\omega|^q \omega) \right\rangle.
\end{aligned}$$
Thus, by \eqref{ls-f}, \eqref{d-f},
\begin{align*}
&2(p-1) \int_{M_+} \varphi |\omega|^{p+q-1}\left\langle \nabla \varphi,
\nabla |\omega| \right\rangle
+ (q+1)(p-1) \int_{M_+} \varphi^2 |\omega|^{p+q-2}|\nabla |\omega||^2\\
\leq&-(p-1)^{2}(A_{p,n,\ell}-1)\int_{M_+}\varphi^2|\omega|^{p+q-2}|\nabla|\omega||^{2}\\
&+ (p-2)q \int_{M_+} \varphi^2 |\omega|^{p+q-2}\left| \nabla
|\omega| \right|^2
 + 2(p-2) \int_{M_+} \varphi |\omega|^{p+q-1}\left| \nabla|\omega| \right| |\nabla \varphi|.
\end{align*}
Hence,
\begin{align*}
(p+q-1+(p-1)^{2}(A_{p,n,\ell}-1)) \int_{M_+} \varphi^2
|\omega|^{p+q-2}|\nabla |\omega||^2 \leq  2(2p-3)\int_{M_+} \varphi
|\omega|^{p+q-1}\left| \nabla|\omega| \right| |\nabla \varphi|.
\end{align*}
Using the fundamental inequality $2AB \leq \varepsilon A^2 +
\varepsilon^{-1}B^2$, we have that, for every $\varepsilon > 0$,
\begin{equation}\label{y-ine}
2 \int_{M_+} \varphi |\omega|^{p+q-1} \left| \nabla |\omega| \right|
|\nabla \varphi|  \leq \varepsilon \int_{M_+} \varphi^2
|\omega|^{p+q-2}|\nabla |\omega||^2 + \dfrac{1}{\varepsilon} \int_{M_+}
|\omega|^{p+q} |\nabla \varphi|^2.
\end{equation}
From the last two inequalities, we obtain
$$
 (p+q-1+(p-1)^{2}(A_{p,n,\ell}-1) - \varepsilon(2p-3)) \int_{M_+} \varphi^2 |\omega|^{p+q-2}|\nabla |\omega||^2
 \leq  \dfrac{2p-3}{\varepsilon}\int_{M_+} |\omega|^{p+q} |\nabla \varphi|^2.
$$
Let $Q:=p+q$, since $Q-1+(p-1)^{2}(A_{p,n,\ell}-1) >0$, we can choose
$\varepsilon > 0$ small enough and a constant $K = K(\varepsilon)
>0$ so that
\begin{equation}\label{e212}
\dfrac{4}{Q^2} \int_{M_+} \varphi^2\left|\nabla
|\omega|^{Q/2}\right|^{2} \leq K\int_{M_+}
|\omega|^{Q}|\nabla\varphi|^2,\quad \text{ for all } r >0.
\end{equation}
Applying a variation of the Duzaar-Fuchs cut-off method (see also
\cite{Duzaar, Nakauchi}), we shall show that \eqref{e212} holds for
every $\varphi\in {C}_0^\infty(M)$. Indeed, we define
$$\eta_{\widetilde{\epsilon}}=\min\left\{\frac{|\omega|}{\widetilde{\epsilon}}, 1\right\}$$
for $\tilde{\epsilon}>0$. Let $\varphi_{\tilde{\epsilon}}=\psi^2\eta_{\tilde{\epsilon}}$, where $\psi\in\mathcal{C}_0^\infty(M)$.
 It is easy to see that $\varphi_{\tilde{\epsilon}}$ is a compactly supported continuous function and
 $\varphi_{\tilde{\epsilon}}=0$ on $M\setminus M_+$. Now, we replace $\varphi$ by $\varphi_{\tilde{\epsilon}}$ in \eqref{e212} and get
\begin{align}
\int_{M_+}&\psi^4(\eta_{\tilde{\epsilon}})^2|\omega|^{Q-2}|\nabla|\omega||^2\notag\\
&\leq
6C\int_{M_+}|\omega|^{Q}|\nabla\psi|^2\psi^2(\eta_{\tilde{\epsilon}})^2+3C\int_{M_+}|\omega|^{Q}|\nabla\eta_{\tilde{\epsilon}}|^2\psi^4.\label{ved1}
\end{align}
Observe that
$$
\int_{M_+}|\omega|^{Q}|\nabla\eta_{\tilde{\epsilon}}|^2\psi^4
\leq
\tilde{\epsilon}^{Q-2}\int_{M_+}|\nabla|\omega||^2\psi^4\chi_{\{|\omega|\leq\tilde{\epsilon}\}}
$$
and the right hand side vanishes by dominated convergence as
$\tilde{\epsilon}\to0$, because $|\nabla|\omega||\in L^2_{loc}(M)$ and $Q\geq2$.
Letting $\tilde{\epsilon}\to 0$ and applying Fatou lemma to the
integral on the left hand side and dominated convergence to the
first integral in the right hand side of \eqref{ved1}, we obtain
\begin{equation}\label{dee212}
\int_{M_+}\psi^4|\omega|^{Q-2}|\nabla|\omega||^2\leq
6C\int_{M_+}|\omega|^{Q}|\nabla\psi|^2\psi^2,
\end{equation}
where $\psi \in C^\infty_0 (M)$.
We choose cutoff functions $\psi \in \mathcal
C^{\infty}_0(M_+)$ 
satisfying
$$
\psi=
\begin{cases}
1, & \text{ on } B_r\\
\in [0,1] \text{ and } |\nabla \varphi| \leq \frac{2}{r}, & \text{ on } B_{2r}\setminus B_r\\
0, & \text{ on } M\setminus B_{2r}
\end{cases},
$$
where $B_r$ is the open ball of radius $r$ and center at a fixed point of $M$.

Letting $r \to \infty$, we conclude that $|\omega|$ is constant on $M_+$. Since $|\omega|=0 \in \partial M_+$, it implies that either $\omega$ is zero or $M_+=\emptyset$. If $M_+=\emptyset$, then $|\omega|$ is contant on $M$. Thanks to the assumption $|\omega|\in L^Q(M)$, we infer $\omega=0$.

$2.$ Assume that
$$
\dfrac{n^2 |H|^2 - 2k}{C_{n, \ell}} < |A|^2 \leq \dfrac{n^2
|H|^2}{C_{n, \ell}}
$$
then $ -k \leq  K_{\ell} <0$. From this and \eqref{m-f} we obtain
\begin{align}
\int_{M_+} \left\langle \nabla(\varphi^2 |\omega|^{q+1}), \nabla
|\omega|^{p-1}\right\rangle \leq
&-(p-1)^{2}(A_{p,n,\ell}-1)\int_{M_+}|\omega|^{p+q-2}|\nabla|\omega||^{2}\notag\\
& \int_{M_+} \left\langle d (|\omega|^{p-2}\omega),
d(\varphi^2|\omega|^q \omega) \right\rangle + k \int_{M_+} \varphi^2
|\omega|^{p+q}.\label{dd}
\end{align}
By the definition of $\lambda_1(M)$ and \eqref{y-ine},
we obtain that, for any $\varepsilon > 0$,
\begin{align}
\lambda_1 \int_{M_+} \varphi^2 |\omega|^{p+q} \leq &\int_{M_+} \left| \nabla \left(\varphi |\omega|^{(p+q)/2}\right) \right|^2 \nonumber \\
 =& \dfrac{(p+q)^2}{4}\int_{M_+} \varphi^2 |\omega|^{p+q-2} \left| \nabla |\omega| \right|^2 + \int_{M_+} |\omega|^{p+q} |\nabla \varphi|^2 \nonumber \\
&+ (p+q)\int_{M_+} \varphi |\omega|^{p+q-1} \left \langle \nabla |\omega|, \nabla \varphi \right \rangle \nonumber \\
& \leq (1+\varepsilon) \dfrac{(p+q)^2}{4} \int_{M_+} \varphi^2
|\omega|^{p+q-2} \left| \nabla |\omega| \right|^2 + \left( 1 +
\dfrac{1}{\varepsilon} \right) \int_{M_+} |\omega|^{p+q} |\nabla
\varphi|^2. \label{c-ie}
\end{align}
From this, \eqref{ls-f}, \eqref{d-f}, and \eqref{dd}, we have
$$
C_{\varepsilon} \int_{M_+} \varphi^2 |\omega|^{p+q-2}|\nabla |\omega||^2
\leq D_{\varepsilon} \int_{M_+} |\omega|^{p+q} |\nabla \varphi|^2,
$$
where
$$\begin{aligned}
C_{\varepsilon}
: &= p+q-1+(p-1)^2(A_{p,n,\ell}-1) - \varepsilon(2p-3)- (1+\varepsilon) \dfrac{k(p+q)^2}{4\lambda_1(M)}\notag\\
&=Q-1+(p-1)^2(A_{p,n,\ell}-1) - \varepsilon(2p-3)- (1+\varepsilon) \dfrac{kQ^2}{4\lambda_1(M)},
\end{aligned}$$
and
$$
D_{\varepsilon}: = \dfrac{2p-3}{\varepsilon} +
\dfrac{k}{\lambda_1(M)}\left( 1 + \dfrac{1}{\varepsilon} \right).
$$ 
Here $Q=p+q\geq2$. Since
 $$
  \lambda_1(M) > \dfrac{kQ^2}{4(Q-1+(p-1)^2(A_{p,n,\ell}-1))},
 $$
there are some small enough number $\varepsilon >0$ and constant $K
= K(\varepsilon)>0$ such that
$$
\dfrac{4}{Q^2} \int_{M_+} \left|\nabla
|\omega|^{Q/2}\right|^{2} \leq K\int_{M_+}
|\omega|^{Q}|\nabla\varphi|^2.
$$
Arguing similarly as in the proof of the first part, we conclude that $|\omega|$ is constant.
Since $\lambda_1(M)>0$, $M$ must have infinite volume, note that $|\omega| \in L^{Q}(M)$, we have that $\omega$ is zero.

$3.$ Assume that $ |A|^2 > \dfrac{n^2 |H|^2}{C_{n, \ell}}$. Then, from
\eqref{m-f} we have
\begin{align}
&\int_{M_+} \left\langle \nabla(\varphi^2 |\omega|^{q+1}), \nabla |\omega|^{p-1}\right\rangle + \dfrac{n^2}{2} \int_{M_+} \varphi^2 |H|^2 |\omega|^{p+q} \notag\\
\leq&-(p-1)^2(A_{p,n,\ell}-1)\int_{M_+}\varphi^2|\omega|^{p+q-2}|\nabla|\omega||^2\notag\\
&+ \int_{M_+} \left\langle d (|\omega|^{p-2}\omega),
d(\varphi^2|\omega|^q \omega) \right\rangle + \dfrac{C_{n, \ell}}{2}
\int_{M_+} \varphi^2 |A|^2 |\omega|^{p+q} + k \int_{M_+} \varphi^2
|\omega|^{p+q}.\label{m3-f}
\end{align}
By H\"older inequality and Lemma \ref{sb-ie}, we have
\begin{align*}
\int_{M_+} \varphi^2 |A|^2 |\omega|^{p+q} & \leq \|A\|_n^2 \left( \int_{M_+} \left(\varphi |\omega|^{(p+q)/2} \right)^{\frac{2n}{n-2}} \right)^{\frac{n-2}{n}} \\
& \leq C_S \|A\|_n^2 \left( \int_{M_+} \left| \nabla \left( \varphi
|\omega|^{(p+q)/2} \right) \right|^2  + \int_{M_+} \varphi^2 |H|^2
|\omega|^{p+q} \right),
\end{align*}
where $C_S$ is the Sobolev constant depending only on $n$. From the
last inequality and \eqref{y-ine} we can get that, for any
$\varepsilon > 0$,
\begin{align*}
\int_{M_+} \varphi^2 |A|^2 |\omega|^{p+q} & \leq C_S \|A\|_n^2 (1+\varepsilon) \dfrac{(p+q)^2}{4} \int_{M_+} \varphi^2 |\omega|^{p+q-2} \left| \nabla |\omega| \right|^2 \\
& + C_S \|A\|_n^2 \left( 1 + \dfrac{1}{\varepsilon} \right) \int_{M_+}
|\omega|^{p+q} |\nabla \varphi|^2 + C_S \|A\|_n^2 \int_{M_+} \varphi^2
|H|^2 |\omega|^{p+q}.
\end{align*}
Using this inequality and \eqref{ls-f}, \eqref{d-f}, \eqref{m3-f},
we have
$$
C_{\varepsilon} \int_{M_+} \varphi^2 |\omega|^{p+q-2}|\nabla |\omega||^2
+ \dfrac{1}{2}\left( n^2 - C_{n, \ell} C_S \|A\|_n^2 \right) \int_{M_+}
\varphi^2 |H|^2 |\omega|^{p+q} \leq D_{\varepsilon} \int_{M_+}
|\omega|^{p+q} |\nabla \varphi|^2,
$$
where, for $Q:=p+q\geq2$
$$\begin{aligned}
C_{\varepsilon}: 
&= p+q-1+(p-1)^2(A_{p,n,\ell}-1) - \varepsilon
(2p-3)- (1+\varepsilon)\dfrac{(p+q)^2}{4}\left(
\dfrac{k}{\lambda_1(M)} + \dfrac{C_{n, \ell}C_S\|A\|_n^2}{2}  \right)\\
&=Q-1+(p-1)^2(A_{p,n,\ell}-1) - \varepsilon
(2p-3)- (1+\varepsilon)\dfrac{Q^2}{4}\left(
\dfrac{k}{\lambda_1(M)} + \dfrac{C_{n, \ell}C_S\|A\|_n^2}{2}  \right),
\end{aligned}$$
and
$$
D_{\varepsilon}: = \dfrac{2p-3}{\varepsilon} + \left( 1 +
\dfrac{1}{\varepsilon} \right) \left( \dfrac{k}{\lambda_1(M)} +
\dfrac{C_{n, \ell}C_S\|A\|_n^2}{2}  \right).
$$
Since
 $$
 \|A\|_n^2 < \min \left\{ \dfrac{n^2}{C_{n, \ell}C_S},\ \dfrac{2}{C_{n, \ell}C_S}\left[ \dfrac{4(Q-1+(p-1)^2(A_{p,n,\ell}-1))}{Q^2} - \dfrac{k}{\lambda_1(M)} \right] \right\},
 $$
there are some small enough $\varepsilon >0$ and constant $K =
K(\varepsilon)>0$ such that
$$
\dfrac{4}{Q^2} \int_{M_+} \left|\nabla
|\omega|^{Q/2}\right|^{2}\varphi^2 \leq K\int_{M_+}
|\omega|^{Q}|\nabla\varphi|^2.
$$
Using the same argument as in the proof of the first and second part, this inequality implies that $\omega$ is zero.

$4.$ Suppose that $\sup_M |A|^2 < \infty$. Then using \eqref{c-ie}
we have that, for any $\varepsilon >0$,
\begin{align*}
\int_{M_+} \varphi^2 |A|^2 |\omega|^{p+q} &\leq \sup_M |A|^2 \int_{M_+} \varphi^2 |\omega|^{p+q} \leq \dfrac{\sup_M |A|^2}{\lambda_1(M)} \int_{M_+} \left| \nabla\left( \varphi |\omega|^{(p+q)/2}\right) \right|^2 \\
& \leq \dfrac{\sup_M |A|^2}{\lambda_1(M)} (1+\varepsilon) \dfrac{(p+q)^2}{4} \int_{M_+} \varphi^2 |\omega|^{p+q-2} \left| \nabla |\omega| \right|^2 \\
& + \dfrac{\sup_M |A|^2}{\lambda_1(M)}  \left( 1 +
\dfrac{1}{\varepsilon} \right) \int_{M_+} |\omega|^{p+q} |\nabla
\varphi|^2.
\end{align*}
From this and \eqref{ls-f}, \eqref{d-f}, \eqref{m3-f} we obtain
that, for any $\varepsilon>0$,
$$
C_{\varepsilon} \int_{M_+} \varphi^2 |\omega|^{p+q-2}|\nabla |\omega||^2
+ \dfrac{n^2}{2} \int_{M_+} \varphi^2 |H|^2 |\omega|^{p+q} \leq
D_{\varepsilon} \int_{M_+} |\omega|^{p+q} |\nabla \varphi|^2,
$$
where
$$\begin{aligned}
C_{\varepsilon}: 
&= p+q-1+(p-1)^2(A_{p,n,\ell}-1) - \varepsilon(2p-3)- (1+\varepsilon)\dfrac{(p+q)^2}{4\lambda_1(M)}\left( k +
\dfrac{C_{n, \ell} \sup_M |A|^2}{2} \right)\\
&=Q-1+(p-1)^2(A_{p,n,\ell}-1) - \varepsilon(2p-3)- (1+\varepsilon)\dfrac{Q^2}{4\lambda_1(M)}\left( k +
\dfrac{C_{n, \ell} \sup_M |A|^2}{2} \right),
\end{aligned}$$
and
$$
D_{\varepsilon}: = \dfrac{2p-3}{\varepsilon} + \left( 1 +
\dfrac{1}{\varepsilon} \right) \dfrac{1}{\lambda_1(M)}\left( k +
\dfrac{C_{n, \ell} \sup_M |A|^2}{2} \right).
$$
Since
 $$
 \lambda_1(M) > \dfrac{Q^2(2k + C_{n, \ell} \sup_M |A|^2)}{8(Q-1+(p-1)^2(A_{p,n,\ell}-1))},
 $$
$M$ must have infinite volume. Moreover, there are some small enough $\varepsilon >0$ and constant $K =
K(\varepsilon)>0$ such that
$$
\dfrac{4}{Q^2} \int_{M_+} \left|\nabla
|\omega|^{Q/2}\right|^{2}\varphi^2 \leq K\int_{M_+}
|\omega|^{Q}|\nabla\varphi|^2.
$$
Using the same argument as in the proof of the first and second part, this inequality implies that $\omega$ is zero.
\end{proof}
Now, we will give a geometric application of Theorem \ref{lin-thm}
that is the Theorem \ref{main1}. First, let us recall the following
result about the existence of $p$-harmonic function on a Riemannian
manifold.
\begin{theorem}[\cite{CCW}]\label{twoends}
Let $M$ be a Riemannian manifold with at least two $p$-nonparabolic
ends. Then, there exists a non-constant, bounded $p$-harmonic
function $u\in{\mathcal C^{1, \alpha}}(M)$ for some $\alpha>0$ such
that $|\nabla u|\in L^p(M)$.
\end{theorem}
Note that, it is known that the regularity of (weakly) $p$-harmonic
function $u$ is not better than ${\mathcal C}^{1, \alpha}_{loc}$
(see \cite{Tolk84} and the references therein). Moreover $u\in
W^{2,2}_{loc}$ if $p\geq2$; $u\in W^{2,p}_{loc}$ if $1<p<2$ (see
\cite{Tolk84}). In fact, any nontrivial (weakly) $p$-harmonic
function $u$ on $M$ is smooth away from the set $\left\{\nabla
u=0\right\}$ which has $n$-dimensional Hausdorff measure zero.
\begin{proof}[Proof of Theorem \ref{main1}] The proof follows by combining Theorem \ref{lin-thm} with $q=0, l=1$ and Theorem \ref{twoends}.
\end{proof}
\section{$p$-harmonic $\ell$-forms on Riemannian manifolds with weighted Poincar\'{e} inequality}
\setcounter{equation}{0}
\begin{lemma}\label{m-lem}
Let $M$ be a complete Riemannian manifold satisfying a weighted
Poincar\'{e} inequality with a continuous weight function $\rho$.
Suppose that $\omega$ is a closed $\ell$-form with finite $L^{p+q}$
norm on $M$ satisfies the following differential inequality
\begin{equation}\label{d-ine}
|\omega| \Delta |\omega|^{p-1} \geq B |\omega|^{p-2} \left| \nabla
|\omega|\right|^2 - \left\langle \delta d (|\omega|^{p-2}\omega),
\omega \right\rangle - a\rho |\omega|^p - b|\omega|^p,
\end{equation}
for some constants $0 < a< \dfrac{4(Q-1+B)}{Q^2}, b>0$ and $Q\geq2$. Then
the following integral inequality holds
\begin{equation}\label{m-ine}
\int_{M_+} \left| \nabla |\omega|^{Q/2} \right|^2 \leq
\dfrac{bQ^2}{4(Q-1+B) - aQ^2} \int_{M_+} |\omega|^{Q}.
\end{equation}
Moreover, if equality holds in \eqref{m-ine}, then equality holds in
\eqref{d-ine}
\end{lemma}

\begin{proof}
(i) Assume that the manifold is compact. Multiplying inequality
\eqref{d-ine} by $|\omega|^q$ and then integrating by parts, we
obtain
\begin{align*}
\int_{M_+}|\omega|^{q+1} \Delta |\omega|^{p-1} \geq &B \int_{M_+}|\omega|^{p+q-2} \left| \nabla |\omega|\right|^2 - \int_{M_+}\left\langle \delta d (|\omega|^{p-2}\omega), |\omega|^q\omega \right\rangle\\
& - a\int_{M_+}\rho |\omega|^{p+q} - b\int_{M_+}|\omega|^{p+q},
\end{align*}
and then,
\begin{align*}
[(q+1)(p-1)+B]\int_{M_+}|\omega|^{p+q-2} \left| \nabla |\omega|\right|^2 \leq& \int_{M_+}\left\langle d (|\omega|^{p-2}\omega), d(|\omega|^q\omega) \right\rangle\\
& + a\int_{M_+}\rho |\omega|^{p+q} + b\int_{M_+}|\omega|^{p+q}.
\end{align*}
Similarly to \eqref{d-f}, we have
\begin{align*}
\int_{M_+} \left\langle d (|\omega|^{p-2}\omega), d(|\omega|^q \omega) \right\rangle  =& \int_{M_+} \left\langle d(|\omega|^{p-2}) \wedge \omega, d(|\omega|^q) \wedge \omega \right\rangle \\
& \leq \int_{M_+} \left|d(|\omega|^{p-2}) \wedge \omega|.|d(|\omega|^q) \wedge \omega \right| \\
& \leq  \int_{M_+} \left| \nabla(|\omega|^{p-2})\right||\omega|.\left|\nabla(|\omega|^q) \right| |\omega| \\
& = (p-2)q \int_{M_+} |\omega|^{p+q-2}\left| \nabla |\omega|
\right|^2.
\end{align*}
By the weighted Poincar\'{e} inequality we have that
$$
\int_{M_+} \rho |\omega|^{p+q} \leq \int_{M_+} \left| \nabla \left(
|\omega|^{(p+q)/2} \right) \right|^2 = \dfrac{(p+q)^2}{4}\int_{M_+}
|\omega|^{p+q-2}\left| \nabla |\omega| \right|^2.
$$
Combining the last three inequalities, we obtain
$$
\left[ p+q-1 + B - a\dfrac{(p+q)^2}{4} \right]
\int_{M_+} |\omega|^{p+q-2}\left| \nabla |\omega| \right|^2 \leq b
\int_{M_+} |\omega|^{p+q},
$$
consequently,
$$
\left[ \dfrac{4(p+q-1+B)}{(p+q)^2} - a \right] \int_{M_+} \left|
\nabla |\omega|^{(p+q)/2} \right|^2 \leq b \int_{M_+} |\omega|^{p+q}.
$$
Let $Q:=p+q$, the last inequality implies that inequality \eqref{m-ine} holds.

Now assume that equality holds in \eqref{m-ine}. Multiplying
inequality \eqref{d-ine} by $|\omega|^q$ where $q=Q-p$, and then integrating by
parts and using the above estimates, we obtain
\begin{align*}
0 \leq& \int_{M_+} \left( |\omega|^{q+1} \Delta |\omega|^{p-1} - B |\omega|^{p+q-2} \left| \nabla |\omega|\right|^2 + \left\langle \delta d (|\omega|^{p-2}\omega), |\omega|^q\omega \right\rangle + a \rho |\omega|^{p+q} + b|\omega|^{p+q} \right)\\
& = -[(q+1)(p-1)+B] \int_{M_+} |\omega|^{p+q-2} \left| \nabla |\omega|\right|^2 + \int_{M_+} \left\langle  d (|\omega|^{p-2}\omega), d(|\omega|^q\omega )\right\rangle \\
& + a \int_{M_+} \rho |\omega|^{p+q} + b\int_{M_+} |\omega|^{p+q} \\
& \leq - \left[ \dfrac{4(p+q-1+B)}{(p+q)^2} - a \right] \int_{M_+} \left| \nabla |\omega|^{(p+q)/2} \right|^2 + b \int_{M_+} |\omega|^{p+q}\\
& = 0.
\end{align*}
Therefore, we can conclude that equality holds in
\eqref{d-ine} in $M_+$. Since in $M\setminus M_+$, \eqref{d-ine} is always true. We complete the proof.

(ii) Assume that the manifold is non-compact. We choose cutoff
function $\varphi \in \mathcal C^{\infty}_0(M_+)$ as in
the proof of Theorem \ref{lin-thm}. Multiplying both sides of
inequality \eqref{d-ine} by $\varphi^2 |\omega|^q$ and then
integrating by parts, we obtain
\begin{align*}
\int_{M_+} \varphi^2 |\omega|^{q+1} \Delta |\omega|^{p-1} \geq &B \int_{M_+} \varphi^2 |\omega|^{p+q-2}\left|\nabla |\omega|\right|^2 - \int_{M_+} \left\langle \delta d (|\omega|^{p-2}\omega), \varphi^2|\omega|^q \omega \right\rangle \\
& - a\int_{M_+} \rho \varphi^2 |\omega|^{p+q} - b \int_{M_+} \varphi^2
|\omega|^{p+q},
\end{align*}
and then,
\begin{align*}
& \int_{M_+} \left\langle \nabla(\varphi^2 |\omega|^{q+1}), \nabla |\omega|^{p-1}\right\rangle + B \int_{M_+} \varphi^2 |\omega|^{p+q-2}\left|\nabla |\omega|\right|^2 \\
& \leq \int_{M_+} \left\langle d (|\omega|^{p-2}\omega),
d(\varphi^2|\omega|^q \omega) \right\rangle + a\int_{M_+} \rho \varphi^2
|\omega|^{p+q} + b \int_{M_+} \varphi^2 |\omega|^{p+q},
\end{align*}
Using the last inequality and \eqref{ls-f}, \eqref{d-f}, we have
\begin{align*}
(p+q-1 + B) \int_{M_+} \varphi^2 |\omega|^{p+q-2}|\nabla |\omega||^2  \leq &a\int_{M_+} \rho \varphi^2 |\omega|^{p+q} + b \int_{M_+} \varphi^2 |\omega|^{p+q} \\
& + 2(2p-3) \int_{M_+} \varphi |\omega|^{p+q-1}\left| \nabla|\omega|
\right| |\nabla \varphi|.
\end{align*}
Similarly to \eqref{c-ie} and by the weighted Poincar\'{e}
inequality we have that
\begin{align*}
\int_{M_+} \rho \varphi^2 |\omega|^{p+q} & \leq \int_{M_+} \left| \nabla \left( \varphi |\omega|^{(p+q)/2} \right) \right|^2 \nonumber \\
& \leq (1+\varepsilon) \dfrac{(p+q)^2}{4} \int_{M_+} \varphi^2
|\omega|^{p+q-2} \left| \nabla |\omega| \right|^2 + \left( 1 +
\dfrac{1}{\varepsilon} \right) \int_{M_+} |\omega|^{p+q} |\nabla
\varphi|^2. \label{p-ine}
\end{align*}

Combining the last two inequalities and using \eqref{y-ine}, we
obtain
\begin{align*}
&\left[ p+q-1 + B - \varepsilon(2p-3) - (1+\varepsilon) \dfrac{a(p+q)^2}{4}  \right] \int_{M_+} \varphi^2 |\omega|^{p+q-2}\left| \nabla |\omega| \right|^2 \\
& \leq b \int_{M_+} \varphi^2 |\omega|^{p+q} + \left[ \dfrac{2p-3}{\varepsilon} + a\left(1 +\dfrac{1}{\varepsilon} \right)
\right]\int_{M_+} |\omega|^{p+q} |\nabla \varphi|^2.
\end{align*}
Since $4(p+q-1+B) - a(p+q)^2 >0$, $p+q-1 + B -
(1+\varepsilon) \dfrac{a(p+q)^2}{4} - \varepsilon(2p-3) >0$ for
all sufficiently small enough $\varepsilon >0$. By the monotone
convergence theorem, letting $r \to \infty$, and then $\varepsilon
\to 0$, we obtain inequality \eqref{m-ine}.

Now suppose that equality in \eqref{m-ine} holds. Multiplying both
sides of inequality \eqref{d-ine} by $\varphi^2 |\omega|^q$ and then
integrating by parts and using the above estimates, we obtain
\begin{align*}
0 \leq& \int_{M_+} \varphi^2\left( |\omega|^{q+1} \Delta |\omega|^{p-1} - B |\omega|^{p+q-2} \left| \nabla |\omega|\right|^2 + \left\langle \delta d (|\omega|^{p-2}\omega), |\omega|^q\omega \right\rangle + a \rho |\omega|^{p+q} + b|\omega|^{p+q} \right)\\
\leq& - [p+q-1 + B] \int_{M_+} \varphi^2 |\omega|^{p+q-2} \left| \nabla |\omega|\right|^2 + 2 (2p-3) \int_{M_+} \varphi |\omega|^{p+q-1}|\nabla |\omega|| \ |\nabla \varphi| \\
& + a \int_{M_+} \rho \varphi^2 |\omega|^{p+q} + b\int_{M_+} \varphi^2 |\omega|^{p+q} \\
\leq & - \left[ p+q-1 + B - \varepsilon(2p-3) - (1+\varepsilon) \dfrac{a(p+q)^2}{4}  \right] \int_{M_+} \varphi^2 |\omega|^{p+q-2}\left| \nabla |\omega| \right|^2 \\
& + b \int_{M_+} \varphi^2 |\omega|^{p+q} + \left[ \dfrac{2p-3}{\varepsilon} + a\left(1 +\dfrac{1}{\varepsilon} \right)
\right]\int_{M_+} |\omega|^{p+q} |\nabla \varphi|^2.
\end{align*}
Letting $r \to \infty$ in the last inequality and using the monotone
convergence theorem, we get
\begin{align*}
0 &\leq \int_{M_+} \left( |\omega|^{q+1} \Delta |\omega|^{p-1} - B |\omega|^{p+q-2} \left| \nabla |\omega|\right|^2 + \left\langle \delta d (|\omega|^{p-2}\omega), |\omega|^q\omega \right\rangle + a \rho |\omega|^{p+q} + b|\omega|^{p+q} \right)\\
&\leq  - \left[ p+q-1 + B - \varepsilon(2p-3) -
(1+\varepsilon) \dfrac{a(p+q)^2}{4}  \right] \int_{M_+}
|\omega|^{p+q-2}\left| \nabla |\omega| \right|^2 + b \int_{M_+}
|\omega|^{p+q}.
\end{align*}
And then putting $\varepsilon \to 0$, we obtain
\begin{align*}
0 &\leq \int_{M_+} \left( |\omega|^{q+1} \Delta |\omega|^{p-1} - B |\omega|^{p+q-2} \left| \nabla |\omega|\right|^2 + \left\langle \delta d (|\omega|^{p-2}\omega), |\omega|^q\omega \right\rangle + a \rho |\omega|^{p+q} + b|\omega|^{p+q} \right)\\
&\leq  - \left[ p+q-1 + B - \dfrac{a(p+q)^2}{4}  \right] \int_{M_+} |\omega|^{p+q-2}\left| \nabla |\omega| \right|^2 + b \int_{M_+} |\omega|^{p+q} \\
& = 0.
\end{align*}
In view of \eqref{d-ine}, we can conclude that equality holds in
\eqref{d-ine}.
\end{proof}

The following result was showed by Vieira in \cite{vieira}.

\begin{lemma}[\cite{vieira}]\label{m1-lem}
Suppose that $u$ is a smooth function on a complete Riemannian
manifold $M$ with finite $L^{Q}$ norm, for $Q\geq2$. Then
$$
\lambda_1(M) \int_{M} |u|^{Q} \leq \int_{M} \left| \nabla
u^{Q/2}\right|^2.
$$
\end{lemma}

\begin{theorem}\label{appl1}
Let $M^n$ be a complete non-compact Riemannian manifold $M^n$
satisfying a weighted Poincar\'{e} inequality with a continuous
weight function $\rho$. Suppose that the curvature operator acting
on $\ell$-forms has a lower bound
$$
K_{\ell} \geq - a \rho - b
$$
for some constants $b>0, Q\geq2$ and 
$$0 < a < \frac{4(Q-1+(p-1)^2
(A_{p, n, \ell} -1))}{Q^2}.$$ 
Assume that the first eigenvalue of the
Laplacian has a lower bound
$$
\lambda_1 (M^n) > \dfrac{bQ^2}{4(Q-1+(p-1)^2 (A_{p, n,
\ell} -1)) - aQ^2}.
$$
Then the space of $p$-harmonic $\ell$-forms with finite $L^Q$ energy on $M^n$ is
trivial.
\end{theorem}
\begin{proof}
Let $\omega$ be any $p$-harmonic $\ell$-form with finite $L^{Q}$
norm. 
From the Bochner formula and the Kato type inequality (see, Lemma
\ref{kt-ine}), we obtain \eqref{duti1}
\begin{equation*}
|\omega| \Delta |\omega|^{p-1} \geq (p-1)^2 (A_{p, n, \ell} -1)
|\omega|^{p-2} \left| \nabla |\omega|\right|^2 - \left\langle \delta
d (|\omega|^{p-2}\omega), \omega \right\rangle + K_{\ell}|\omega|^p.
\end{equation*}
Hence,
\begin{equation*}
|\omega| \Delta |\omega|^{p-1} \geq (p-1)^2 (A_{p, n, \ell} -1)
|\omega|^{p-2} \left| \nabla |\omega|\right|^2 - \left\langle \delta
d (|\omega|^{p-2}\omega), \omega \right\rangle - a \rho |\omega|^p -
b |\omega|^p.
\end{equation*}
Applying Lemma \ref{m-lem} to $B =(p-1)^2 (A_{p, n, \ell} -1)$, we
obtain
$$
\int_{M_+} \left| \nabla |\omega|^{Q/2} \right|^2 \leq
\dfrac{bQ^2}{4(Q-1+(p-1)^2 (A_{p, n, \ell} -1)) - aQ^2}
\int_{M_+} |\omega|^{Q}.
$$
Since $|\omega|\in L^Q(M)$ this implies $|\omega|\in L^Q(M_+)$. Therefore, by Lemma \ref{m1-lem}, we have
$$
\lambda_1(M^n) \int_{M_+} |\omega|^{Q} \leq \int_{M_+} \left| \nabla
\omega^{Q/2}\right|^2.
$$
From the last two inequalities, we obtain
$$
\lambda_1(M^n) \int_{M_+} |\omega|^{Q} \leq \dfrac{bQ^2}{4(Q-1+(p-1)^2 (A_{p, n, \ell} -1)) - aQ^2} \int_{M_+} |\omega|^{Q}.
$$
If the $\ell$-form $\omega$ is not identically zero in $M_+$ (therefore $\omega=0$ in $M$), then
$$
\lambda_1 (M^n) \leq \dfrac{bQ^2}{4(Q-1+(p-1)^2 (A_{p,
n, \ell} -1)) - aQ^2},
$$
which leads to a contradiction. So $\omega$ is identically zero.
\end{proof}
Combining Theorem \ref{twoends} and Theorem \ref{appl1} with $Q=p\geq2,
\ell=1$, we obtain the follows result.
\begin{corollary}
Let $M^n$ be a complete non-compact Riemannian manifold $M^n$
satisfying a weighted Poincar\'{e} inequality with a continuous
weight function $\rho$. Suppose that
$$
Ric_M \geq - a \rho - b, \quad \text{ for }0 < a <
\dfrac{4(p-1)(p+n-2)}{p^2(n-1)}
$$
and some constants $b>0$. Assume that the first eigenvalue of the
Laplacian has a lower bound
$$
\lambda_1 (M^n) > \dfrac{bp^2(n-1)}{4(p-1)(n+p-2) - ap^2(n-1)}.
$$
Then the space of $L^{p}$ $p$-harmonic $1$-forms on $M^n$ is
trivial. Therefore, $M$ has at most one $p$-nonparabolic end.
\end{corollary}
Let
$$ A_{n, \ell}=\begin{cases}
\dfrac{n-\ell+1}{n-\ell},&\text{if }1\leq \ell\leq\frac{n}{2}\\
\dfrac{\ell+1}{\ell},&\text{if }\frac{n}{2}\leq \ell\leq n-1.
\end{cases} $$
We conclude this section by the below rigidity property.
\begin{corollary}
Let $M^n$ be a complete non-compact Riemannian manifold $M^n$
satisfying a weighted Poincar\'{e} inequality with a continuous
weight function $\rho$. Suppose that
$$
Ric_M \geq - a \rho - b, \quad \text{ for }0 < a <  A_{n,\ell}
$$
and some constants $b>0$. Assume that
$$
\lambda_1 (M^n) =\frac{b}{A_{n,\ell}-a}.
$$
Then either
\begin{enumerate}
\item The space of $L^{2}$ harmonic $\ell$-forms on $M^n$ is trivial or;
\item  For any $L^2$ harmonic $\ell$-form $\omega$ on $M$, there exixts a $1$-form $\alpha$ such that
$$ \nabla\omega=\alpha\otimes\omega-\frac{1}{\sqrt[]{\ell+1}}\theta_1(\alpha\wedge\omega)+\frac{1}{\sqrt[]{n+1-\ell}}\theta_2(i_\alpha\omega). $$
\end{enumerate}
\end{corollary}
\begin{proof}Let $p=2, q=0, Q=2$, the Bochner formula applying on harmonic $\ell$-form $\omega$ implies that
$$ |\omega|\Delta|\omega|\geq (A_{n,\ell}-1)|\nabla|\omega||^2-a\rho|\omega|^2-b|\omega|^2. $$
If $\omega$ is non-trivial then Lemma \ref{m-lem} and Lemma
\ref{m1-lem} imply that
$$ \lambda_1(M)\int_{M_+}|\omega|^2=\int_{M_+}|\nabla|\omega||^2. $$
Therefore,
$$  |\omega|\Delta|\omega|= (A_{n,\ell}-1)|\nabla|\omega||^2-a\rho|\omega|^2-b|\omega|^2 $$
on $M_+$, hence it holds true on $M$. This means that the equality in the Kato type inequality
\eqref{kato} holds true. By Lemma \ref{kt-ine}, we are done.
\end{proof}
\section{$p$-harmonic $1$-forms on locally conformally flat Riemannian manifolds}
\setcounter{equation}{0} In this section we prove vanishing theorem
for $p$-harmonic $1$-forms $(Q\geq2)$ with finite $L^Q$ energy.
We will use the following auxiliary lemmas.

It is known that a simply connected, locally conformally flat
manifold $M^n,\ n \geq 3$, has a conformal immersion into
$\mathbb{S}^n$, and according to \cite{SchoenYau1994}, the Yamabe
constant of $M^n$ satisfies
$$
\mathcal{Y}(M^n) = \mathcal{Y}(\mathbb{S}^n) = \dfrac{n(n-2)\omega_n^{2/n}}{4},
$$
where $\omega_n$ is the volume of the unit sphere in $\mathbb{R}^n$.
Therefore the following inequality
\begin{equation}\label{ya-ine}
\mathcal{Y}(\mathbb{S}^n)\left( \int_{M} f^{\frac{2n}{n-2}}dv
\right)^{\frac{n-2}{n}} \leq \int_{M} |\nabla f|^2 dv +
\dfrac{n-2}{4(n-1)}\int_{M} R f^2 dv
\end{equation}
holds for all $f \in C_0^{\infty}(M)$. To derive vanishing property of $p$-harmonic $1$-forms, we need to have following lemmas.
\begin{lemma}\label{sb-lem-1} (\cite{Lin15b})
Let $(M^n,g),\ n \geq 3$, be an $n$-dimensional complete, simply
connected, locally conformally flat Riemannian manifold with $R \leq
0$  or $ \|R\|_{n/2} < \infty$. Then the following $L^2$ Sobolev
inequality
$$
\left( \int_{M} f^{\frac{2n}{n-2}}dv \right)^{\frac{n-2}{n}} \leq S
\int_{M} |\nabla f|^2 dv, \quad \forall f \in C_0^{\infty}(M),
$$
holds for some constant $S>0$, which is equal to
$\mathcal{Y}(\mathbb{S}^n)^{-1}$ in the case of $R \leq 0$. In particular, $M$
has infinite volume.
\end{lemma}

\begin{lemma}\label{ric-lem}(\cite{Lin15b})
Let $(M^n,g)$ be an $n$-dimensional complete Riemannian manifold.
Then
$$
\Ric \geq - |T|g - \dfrac{|R|}{\sqrt{n}}g
$$
in the sense of quadratic forms. Here $T$ stands for the traceless tensor, namely
$$T=\operatorname{Ric}-\frac{R}{n}g.$$
\end{lemma}
Now, we introduce our result.
\begin{theorem}\label{lin-thm2}
Let $(M^n,g), \ n \geq 3$, be an $n$-dimensional complete, simply
connected, locally conformally flat Riemannian manifold. If one of
the following conditions
\begin{itemize}
\item[1.]
\begin{equation*}
 \|T\|_{n/2} + \dfrac{\|R\|_{n/2}}{\sqrt{n}} < \dfrac{4(Q-1 + \kappa_p)}{SQ^2}.
\end{equation*}
\item[2.] the scalar curvature $R$ is nonpositive and
$$
K_{p,Q,n}: = \dfrac{Q-1+ \kappa_p}{Q^2} -
\dfrac{n-1}{\sqrt n(n-2)} > 0
$$
and
$$
\|T\|_{n/2} < \dfrac{4K_{p,Q,n}}{S}=4K_{p,Q,n}\mathcal{Y}(\mathbb{S}^n).
$$
\end{itemize}
holds true, then every $p$-harmonic $1$-form with finite $L^{Q} (Q\geq2)$
norm on $M$ is trivial. Here 
$$\kappa_p=\min\left\{1, \frac{(p-1)^2}{n-1}\right\}.$$
\end{theorem}
\begin{proof}Since $M$ satisfies a Sobolev inequality, $M$ must have infinite volume. Therefore, as in the previous part, we only need to prove that if $\omega$ is any $p$-harmonic $1$-form with finite $L^{Q}$-norm then $\omega=0$ in $M_+$.
Now, applying the Bochner formula to the form $|\omega|^{p-2}\omega$, we
have
\begin{align*}
\dfrac{1}{2} \Delta |\omega|^{2(p-1)}  = &\left| \nabla \left( |\omega|^{p-2}\omega \right) \right|^2 - \left \langle \delta d (|\omega|^{p-2}\omega), |\omega|^{p-2}\omega \right \rangle \\
& +  |\omega|^{2(p-2)} \Ric(\omega,\omega) .
\end{align*}
From this and Lemma \ref{ric-lem} and the Kato type inequality (see,
Lemma \ref{kt-ine}), we obtain
\begin{equation*}
|\omega| \Delta |\omega|^{p-1} \geq \kappa_p
|\omega|^{p-2} \left| \nabla |\omega|\right|^2 - \left\langle \delta
d (|\omega|^{p-2}\omega), \omega \right\rangle - |T| |\omega|^p
-\dfrac{|R|}{\sqrt{n}} |\omega|^p.
\end{equation*}
We choose cutoff function $\varphi\in \mathcal
C^{\infty}_0(M_+)$ as in the proof of Theorem \ref{lin-thm}.
Multiplying both sides of the last inequality by $\varphi^2
|\omega|^q, (q=Q-p)$ and then integrating by parts, we obtain
\begin{align*}
\int_{M_+} \varphi^2 |\omega|^{q+1} \Delta |\omega|^{p-1}  
\geq 
&\kappa_p \int_{M_+} \varphi^2 |\omega|^{p+q-2}\left|\nabla |\omega|\right|^2 - \int_{M_+} \left\langle \delta d (|\omega|^{p-2}\omega), \varphi^2|\omega|^q \omega \right\rangle \\
& - \int_{M_+} |T| \varphi^2 |\omega|^{p+q} - \dfrac{1}{\sqrt{n}} \int_{M_+}
|R| \varphi^2 |\omega|^{p+q},
\end{align*}
and then,
\begin{align*}
& \int_{M_+} \left\langle \nabla(\varphi^2 |\omega|^{q+1}), \nabla |\omega|^{p-1}\right\rangle + \kappa_p \int_{M_+} \varphi^2 |\omega|^{p+q-2}\left|\nabla |\omega|\right|^2 \\
& \leq \int_{M_+} \left\langle d (|\omega|^{p-2}\omega),
d(\varphi^2|\omega|^q \omega) \right\rangle + \int_{M_+} |T| \varphi^2
|\omega|^{p+q} + \dfrac{1}{\sqrt{n}} \int_{M_+}  |R| \varphi^2
|\omega|^{p+q}.
\end{align*}
By the hypotheses and Lemma \ref{sb-lem-1}, we obtain that, for some
constant $S>0$ and all functions $f \in C_0^{\infty}(M)$,
$$
\left( \int_{M_+} f^{\frac{2n}{n-2}}dv \right)^{\frac{n-2}{n}} \leq S
\int_{M_+} |\nabla f|^2 dv.
$$
Using this and \eqref{c-ie}, we have that, for any $\varepsilon >0$,
\begin{align}
\int_{M_+} |T| \varphi^2 |\omega|^{p+q}
\leq &\|T\|_{n/2} \left( \int_{M_+} \left(\varphi |\omega|^{(p+q)/2} \right)^{\frac{2n}{n-2}} \right)^{\frac{n-2}{n}} \notag\\
 \leq &S \|T\|_{n/2} \int_{M_+} \left| \nabla \left( \varphi |\omega|^{(p+q)/2} \right) \right|^2 \nonumber \\
 \leq &S \|T\|_{n/2} (1+\varepsilon) \dfrac{(p+q)^2}{4} \int_{M_+} \varphi^2 |\omega|^{p+q-2} \left| \nabla |\omega| \right|^2 \nonumber \\
& + S \|T\|_{n/2}\left( 1 + \dfrac{1}{\varepsilon} \right) \int_{M_+}
|\omega|^{p+q} |\nabla \varphi|^2. \label{T-ine}
\end{align}
From \eqref{ls-f}, \eqref{d-f}, \eqref{y-ine} and \eqref{T-ine} we
obtain that, for any $\varepsilon >0$,
\begin{align}
& \left[ +q-1+\kappa_p-  \varepsilon(2p-3) - S \|T\|_{n/2} (1+\varepsilon) \dfrac{(p+q)^2}{4} \right] \int_{M_+} \varphi^2 |\omega|^{p+q-2}|\nabla |\omega||^2 \nonumber \\
& \leq \left[ \dfrac{\kappa_p}{\varepsilon} + S \|T\|_{n/2}\left( 1 +
\dfrac{1}{\varepsilon} \right) \right] \int_{M_+} |\omega|^{p+q} |\nabla
\varphi|^2 + \dfrac{1}{\sqrt{n}} \int_{M_+} \varphi^2 |R|
|\omega|^{p+q}. \label{m4-f}
\end{align}

For the last term in the right hand side of \eqref{m4-f}, we can
estimate in the following two ways.

$1.$ Similarly to \eqref{T-ine}, we have that, for any $\varepsilon
>0$,
\begin{align*}
\int_{M_+} |R| \varphi^2 |\omega|^{p+q}
& \leq S \|R\|_{n/2} (1+\varepsilon) \dfrac{(p+q)^2}{4} \int_{M_+} \varphi^2 |\omega|^{p+q-2} \left| \nabla |\omega| \right|^2 \\
& + S \|R\|_{n/2}\left( 1 + \dfrac{1}{\varepsilon} \right) \int_{M_+}
|\omega|^{p+q} |\nabla \varphi|^2.
\end{align*}
Consequently, for any $\varepsilon >0$,
$$
C_{\varepsilon} \int_{M_+} \varphi^2 |\omega|^{p+q-2}|\nabla |\omega||^2
\leq D_{\varepsilon} \int_{M_+} |\omega|^{p+q} |\nabla \varphi|^2,
$$
where
$$\begin{aligned}
C_{\varepsilon}: 
&=p+q-1+ \kappa_p -  \varepsilon(2p-3) - S (1+\varepsilon) \dfrac{(p+q)^2}{4} \left( \|T\|_{n/2} +
\dfrac{\|R\|_{n/2}}{\sqrt{n}} \right)\\
&=Q-1 + \kappa_p-  \varepsilon(2p-3)
 - S (1+\varepsilon) \dfrac{Q^2}{4} \left( \|T\|_{n/2} +
\dfrac{\|R\|_{n/2}}{\sqrt{n}} \right)
\end{aligned}$$
and
$$
D_{\varepsilon}: = \dfrac{2p-3}{\varepsilon} + S\left( 1 +
\dfrac{1}{\varepsilon} \right)\left( \|T\|_{n/2} +
\dfrac{\|R\|_{n/2}}{\sqrt{n}} \right).
$$
Since
$$
\|T\|_{n/2} + \dfrac{\|R\|_{n/2}}{\sqrt{n}} < \dfrac{4(Q-1+
(p-1)^2/(n-1))}{SQ^2},
$$
there are some small enough $\varepsilon >0$ and constant $K =
K(\varepsilon)>0$ such that
$$
\dfrac{4}{Q^2} \int_{M_+} \left|\nabla
|\omega|^{Q/2}\right|^{2}\varphi^2 \leq K\int_{M_+}
|\omega|^{Q}|\nabla\varphi|^2.
$$
Using the same argument as in the proof of the first part of Theorem \ref{lin-thm}, this inequality implies that $\omega$ is zero.

$2.$ If the scalar curvature $R$ is nonpositive, then from
\eqref{ya-ine} we have that
$$
\int_{M_+} |R|f^2 dv \leq \dfrac{4(n-1)}{n-2}\int_{M_+} |\nabla f|^2 dv,\
\forall f \in C_0^{\infty}(M).
$$
From this and \eqref{c-ie}, we have that, for any $\varepsilon >0$,
\begin{align*}
\int_{M_+} |R|\varphi^2 |\omega|^{p+q}  \leq &\dfrac{4(n-1)}{n-2}\int_{M_+} |\nabla (\varphi |\omega|^{(p+q)/2})|^2 \\
\leq &\dfrac{4(n-1)}{n-2} (1+\varepsilon) \dfrac{(p+q)^2}{4} \int_{M_+} \varphi^2 |\omega|^{p+q-2} \left| \nabla |\omega| \right|^2 \\
& + \dfrac{4(n-1)}{n-2} \left( 1 + \dfrac{1}{\varepsilon} \right)
\int_{M_+} |\omega|^{p+q} |\nabla \varphi|^2.
\end{align*}
Substituting this inequality into \eqref{m4-f} yields that, for any
$\varepsilon >0$,
$$
C_{\varepsilon} \int_{M_+} \varphi^2 |\omega|^{p+q-2}|\nabla |\omega||^2
\leq D_{\varepsilon} \int_{M_+} |\omega|^{p+q} |\nabla \varphi|^2,
$$
where
$$\begin{aligned}
C_{\varepsilon}: 
&=p+q-1+ \kappa_p -  \varepsilon(2p-3) - (1+\varepsilon) \dfrac{(p+q)^2}{4} \left( S\|T\|_{n/2} +
\dfrac{4(n-1)}{(n-2)\sqrt{n}} \right)\\
&=Q-1 + \kappa_p -  \varepsilon(2p-3)- (1+\varepsilon) \dfrac{Q^2}{4} \left( S\|T\|_{n/2} +
\dfrac{4(n-1)}{(n-2)\sqrt{n}} \right),
\end{aligned}$$
and
$$
D_{\varepsilon}: = \dfrac{2p-3}{\varepsilon} + \left( 1 +
\dfrac{1}{\varepsilon} \right)\left( S\|T\|_{n/2} +
\dfrac{4(n-1)}{(n-2)\sqrt{n}} \right).
$$
Since
$$
K_{p,Q,n}: = \dfrac{Q-1+\kappa_p}{Q^2} -
\dfrac{n-1}{\sqrt n(n-2)} > 0
$$
and
$$
\|T\|_{n/2} < \dfrac{4K_{p,Q,n}}{S},
$$
there are some small enough $\varepsilon >0$ and constant $K =
K(\varepsilon)>0$ such that
$$
\dfrac{4}{Q^2} \int_{M_+} \left|\nabla
|\omega|^{Q/2}\right|^{2}\varphi^2 \leq K\int_{M_+}
|\omega|^{Q}|\nabla\varphi|^2.
$$
As in previous part, we infer that $\omega=0$. The proof is complete.
\end{proof}
Note that when $q=0$, we recover Theorem 1.3 and Theorem 1.4 in
\cite{Lin15b}. Hence, Theorem \ref{lin-thm2} can be considered as a
generalization of Lin's results to the nonlinear setting. Now, we
give a simple proof of Theorem \ref{main2}.
\begin{proof}[Proof of Theorem \ref{main2}] The proof follows by combining Theorem \ref{twoends} and Theorem \ref{lin-thm2} with $Q=p$.
\end{proof}

\vskip 0.3cm \noindent {\bf Acknowledgment:} 
A part of this paper was written during a stay of the authors at
Vietnam Institute for Advanced Study in Mathematics (VIASM) and a
stay of the first author at Institut Fourier, UMR 5582, Grenoble. We
would to express our thanks to the staff there for the hospitality.


\end{document}